\documentclass[11pt]{amsart}

\ifx\MyBeamer\undefined
\usepackage[breaklinks,colorlinks]{hyperref}
\usepackage[a4paper,margin=2cm,includeheadfoot]{geometry}
\else
\if\MyBeamer t
\usepackage[breaklinks,colorlinks]{hyperref}
\usepackage[a4paper,margin=2cm]{geometry}
\else
\def\emph{\alert}
\fi
\fi

\usepackage{mathpazo}
\usepackage{amsmath}
\usepackage{amssymb}
\usepackage{mathrsfs}
\usepackage{amsthm}
\usepackage{aliascnt}
\usepackage[utf8]{inputenc}
\usepackage[T1]{fontenc}
\ifx\FrenchText\undefined
\usepackage[british]{babel}
\else
\usepackage[british,french]{babel}
\fi

\makeatletter

\newcounter{quotethmcnt}

\newcommand{\mynewop}[1]{
  \expandafter\DeclareMathOperator\csname #1\endcsname{#1}
}

\ifx\MyBeamer\undefined

\def\equationautorefname~#1\null{(#1)}
\def\itemautorefname~#1\null{#1}
\fi

\ifx\MyBeamer\undefined

\ifx\thmnum\undefined
\newtheorem{thmnum}{}[section]
\fi

\newcommand{\mynewthm}[3][]{
  \newaliascnt{#2}{thmnum}
  \newtheorem{#2}[#2]{#3}
  \aliascntresetthe{#2}
  \newtheorem*{#2*}{#3}
  \expandafter\newcommand\csname #2autorefname\endcsname{#3}
  \expandafter\renewcommand\csname the#2\endcsname{\thethmnum}
}

\newtheorem*{clm}{Claim}
\newenvironment{clmprf}{
  \begin{proof}[Proof of claim]
    
  }{\end{proof}}

\else

\let\xxx=\frametitle
\def\frametitle#1{
  \xxx{
    \setbeamercolor*{math text}{use={titlelike,my math text},fg=titlelike.fg!80!my math text.fg}
    #1}
  \setbeamercolor{math text}{use=my math text,fg=my math text.fg}
}

\newcommand{\beamerenv}[3]{
\newenvironment<>{#1}
{
  \setbeamercolor{temp}{fg=structure.fg}
  \setbeamercolor{structure}{fg=#2}
  \setbeamercolor{block body}{use=structure,bg=structure.fg!5!white}
  \begin{#3}
}
{\end{#3}\setbeamercolor{structure}{fg=temp.fg}}}

\newcommand{\mynewthm}[3][green!50!black]{
  \newtheorem*{#2x}{#3}
  \beamerenv{#2}{#1}{#2x}
}

\beamerenv{other}{violet}{block}

\usepackage[displaymath,textmath,sections,graphics]{preview}
\PreviewEnvironment*{frame}
\PreviewEnvironment*{other}

\fi

\ifx\FrenchText\undefined
\newcommand{\myiffrench}[2]{#2}
\else
\newcommand{\myiffrench}[2]{\iflanguage{french}{#1}{#2}}
\fi

\theoremstyle{plain}
\mynewthm[red]{thm}{\myiffrench{Théorème}{Theorem}}
\mynewthm{prp}{Proposition}
\mynewthm[purple]{lem}{\myiffrench{Lemme}{Lemma}}
\mynewthm[orange]{fct}{\myiffrench{Fait}{Fact}}
\mynewthm[purple!50!white]{cor}{\myiffrench{Corollaire}{Corollary}}

\theoremstyle{definition}
\mynewthm[green!80!black]{dfn}{\myiffrench{Définition}{Definition}}
\mynewthm{hyp}{\myiffrench{Hypothèse}{Hypothesis}}
\mynewthm{conv}{Convention}
\mynewthm{conj}{Conjecture}
\mynewthm{ntn}{Notation}
\mynewthm{cst}{Construction}

\theoremstyle{remark}
\mynewthm[yellow!90!black]{rmk}{\myiffrench{Remarque}{Remark}}
\mynewthm{qst}{Question}
\mynewthm{exm}{\myiffrench{Exemple}{Example}}

\ifx\MyBeamer\undefined

\newcommand{\myenumlabel}[1]{\textnormal{(\roman{#1})}}

\fi

\def\book{

\def\@thm##1##2##3{
  \ifhmode\unskip\unskip\par\fi
  \normalfont
  \trivlist
  \let\thmheadnl\relax
  \let\thm@swap\@gobble
  \let\thm@indent\noindent 
  \thm@headfont{\bfseries}
  \thm@notefont{\fontseries\mddefault\upshape}
  \thm@headpunct{.}
  \thm@headsep 5\p@ plus\p@ minus\p@\relax
  \thm@space@setup
  ##1
  \@topsep \thm@preskip               
  \@topsepadd \thm@postskip           
  \def\@tempa{##2}\ifx\@empty\@tempa
    \def\@tempa{\@oparg{\@begintheorem{##3}{}}[]}
  \else
    \refstepcounter{##2}
    \def\@tempa{\@oparg{\@begintheorem{##3}{\csname the##2\endcsname}}[]}
  \fi
  \@tempa
}

\renewenvironment{proof}[1][\proofname]{\par
  \pushQED{\qed}
  \normalfont \topsep6\p@\@plus6\p@\relax
  \trivlist
  \item[\hskip\labelsep
        \itshape
    ##1\@addpunct{.}]\ignorespaces
}{
  \popQED\endtrivlist\@endpefalse
}

\newcommand{\excautorefname}{\myiffrench{Exercice}{Exercise}}

}

\newcounter{cycprfcnt}
\newcounter{cycprffirst}
\newcommand{\cycprfpreamble}[1]
{
  \topsep\z@
  \@trivlist
  \def\@itemlabel{\impnext}
  \setcounter{cycprfcnt}{1}
  \setcounter{cycprffirst}{#1}
  \setlength{\itemindent}{\labelsep}
  \def\makelabel##1{\ifnum\value{cycprffirst}=0\setcounter{cycprffirst}{1}\else\indent\fi##1}
  \newcommand{\cpcurr}{\myenumlabel{cycprfcnt}}
  \newcommand{\cpnext}{\addtocounter{cycprfcnt}{1}\cpcurr}
  \newcommand{\cpprev}{\addtocounter{cycprfcnt}{-1}\cpcurr}
  \newcommand{\cpnum}[1]{\setcounter{cycprfcnt}{##1}\cpcurr}
  \newcommand{\cpfirst}{\cpnum{1}}
  \newcommand{\impab}[3][$\Longrightarrow$]{##2 ##1 ##3.}
  \newcommand{\impabc}[4][$\Longrightarrow$]{##2 ##1 ##3 ##1 ##4.}
  \newcommand{\impabd}[5][$\Longrightarrow$]{##2 ##1 ##3 ##1 ##4 ##1 ##5.}
  \newcommand{\impnext}{\impab\cpcurr\cpnext}
  \newcommand{\impnnext}{\impabc\cpcurr\cpnext\cpnext}
  \newcommand{\impnnnext}{\impabcd\cpcurr\cpnext\cpnext\cpnext}
  \newcommand{\impprev}{\impab\cpcurr\cpprev}
  \newcommand{\impfirst}{\impab\cpcurr\cpfirst}
  \newcommand{\impnfirst}{\impabc\cpcurr\cpnext\cpfirst}
  \newcommand{\eqnext}{\impab[$\Longleftrightarrow$]\cpcurr\cpnext}
  \newcommand{\eqnnext}{\impabc[$\Longleftrightarrow$]\cpcurr\cpnext\cpnext}
  \newcommand{\eqnnnext}{\impabcd[$\Longleftrightarrow$]\cpcurr\cpnext\cpnext\cpnext}
  \newcommand{\impnum}[2]{\impab{\cpnum{##1}}{\cpnum{##2}}}
  \newcommand{\impnumnum}[3]{\impab\cpnum{##1}{\cpnum{##2}}{\cpnum{##2}}}
  \newcommand{\eqnum}[2]{\impab[$\Longleftrightarrow$]{\cpnum{##1}}{\cpnum{##2}}}
  \newcommand{\impref}[3][]{\impab{\ref{##1##2}}{\ref{##1##3}}}
}

\newenvironment{cycprf}[1][0]
{\cycprfpreamble{#1}}
{\qedhere\endtrivlist}

\newenvironment{cycprf*}[1][0]
{\cycprfpreamble{#1}}
{\endtrivlist}

\def\indsym#1#2{
  \setbox0=\hbox{$\m@th#1x$}
  \kern\wd0
  \hbox to 0pt{\hss$\m@th#1\mid$\hbox to 0pt{$\m@th#1^{#2}$\hss}\hss}
  \lower.9\ht0\hbox to 0pt{\hss$\m@th#1\smile$\hss}
  \kern\wd0
}

\def\nindsym#1#2{
  \setbox0=\hbox{$\m@th#1x$}
  \kern\wd0
  \hbox to 0pt{\hss$\m@th#1\not$\kern1.4\wd0\hss}
  \hbox to 0pt{\hss$\m@th#1\mid$\hbox to 0pt{$\m@th#1^{#2}$\hss}\hss}
  \lower.9\ht0\hbox to 0pt{\hss$\m@th#1\smile$\hss}
  \kern\wd0
}

\def\dotminussym#1{
  \setbox0=\hbox{$\m@th#1-$}
  \kern.5\wd0
  \hbox to 0pt{\hss\hbox{$\m@th#1-$}\hss}
  \raise.6\ht0\hbox to 0pt{\hss$\m@th#1.$\hss}
  \kern.5\wd0
}

\renewcommand{\emptyset}{\varnothing}
\renewcommand{\setminus}{\smallsetminus}

\mynewop{tp}
\mynewop{qftp}

\mynewop{stp}
\mynewop{Ltp}

\mynewop{Th}
\mynewop{Mod}

\mynewop{dcl}
\mynewop{acl}
\mynewop{bdd}
\mynewop{SU}
\mynewop{cf}
\mynewop{Span}
\mynewop{cl}
\mynewop{diam}
\mynewop{co}

\mynewop{Av}

\mynewop{ob}
\mynewop{Funct}
\mynewop{coker}

\mynewop{Hom}
\mynewop{Emb}
\mynewop{Iso}
\mynewop{Aut}
\mynewop{Autf}
\mynewop{Homeo}

\mynewop{Gal}
\mynewop{End}
\mynewop{dom}
\mynewop{img}
\mynewop{supp}
\mynewop{Pert}

\mynewop{Spec}
\mynewop{Diag}
\mynewop{Frac}

\mynewop{CB}
\mynewop{RM}
\mynewop{dM}
\mynewop{rk}
\mynewop{trdeg}

\mynewop{Stab}

\mynewop{sgn}
\mynewop{med}

\mynewop{st}

\mynewop{vol}
\mynewop{Ann}

\DeclareRobustCommand\lcm{\mathop{\operator@font lcm}}
\renewcommand{\Re}{\mathop{\operator@font Re}}

\newcommand{\cM}{\mathcal{M}}

\makeatother

\begin{document}

\title{A topometric Effros theorem}

\author{Itaï \textsc{Ben Yaacov}}

\address{Itaï \textsc{Ben Yaacov} \\
  Université Claude Bernard -- Lyon 1 \\
  Institut Camille Jordan, CNRS UMR 5208 \\
  43 boulevard du 11 novembre 1918 \\
  69622 Villeurbanne Cedex \\
  France}

\urladdr{\url{http://math.univ-lyon1.fr/~begnac/}}

\author{Julien \textsc{Melleray}}
\address{Julien \textsc{Melleray} \\
  Université Claude Bernard -- Lyon 1 \\
  Institut Camille Jordan, CNRS UMR 5208 \\
  43 boulevard du 11 novembre 1918 \\
  69622 Villeurbanne Cedex \\
  France}
\urladdr{\url{http://math.univ-lyon1.fr/~melleray/}}

\thanks{Research supported by ANR project AGRUME (ANR-17-CE40-0026)}
\thanks{The authors wish to thank Gianluca Basso for helpful comments}


\begin{abstract}
  Given a continuous and isometric action of a Polish group $G$ on an adequate Polish topometric space $(X,\tau,\rho)$ and $x \in X$, we find a necessary and sufficient condition for $\overline{Gx}^\rho$ to be co-meagre; we also obtain a criterion that characterizes when such a point exists.
  This work completes a criterion established in earlier work of the authors.
\end{abstract}

\maketitle

\tableofcontents

\section{Introduction}

Our work in this article is concerned with \emph{Polish topometric spaces}, namely objects of the form $(X,\tau,\rho)$, where $(X,\tau)$ is a Polish topological space and $\rho$ is a lower semi-continuous distance which refines $\tau$. Type spaces in continous logic provide fundamental examples, though our main motivation comes from another direction: given a Polish group $(G,\tau)$, and a left-invariant distance $d$ inducing $\tau$, there is a natural topometric structure obtained by setting $\rho(g,h)= \sup_{k \in G} d(gk,hk)$.
The triplet $(G,\tau,\rho)$ is then a Polish topometric group, i.e.~a Polish space enriched with a topometric structure for which the distance is translation-invariant. These objects played an important part in \cite{BenYaacov-Berenstein-Melleray:TopometricGroups}, and some of their properties were further studied in \cite{BenYaacov-Melleray:Grey}.

An interesting phenomenon was observed in \cite{BenYaacov-Berenstein-Melleray:TopometricGroups}: given an action of a Polish group $G$ on a Polish topometric space $(X,\tau,\rho)$, it might happen that each $G$-orbit is meagre (in the sense of Baire category), yet there exists $x \in X$ such that the $\rho$-closure $\overline{Gx}^\rho$ is co-meagre.
When $\overline{Gx}^\rho$ is co-meagre, we say that $x$ is a \emph{metrically generic element}.
It is of interest, in some concrete cases, to determine precisely what these elements are; see for instance the recent work \cite{Berenstein-Ibarlucia-Henson:ECMPAFG} which interweaves some model theory and ergodic theory.
Under the additional hypothesis of \emph{adequacy} (see \autoref{dfn:AdequateDistance} below; this assumption is satisfied by both kinds of Polish topometric spaces we mentioned above), it was proved in \cite{BenYaacov-Melleray:Grey} that metrically generic elements form a $G_\delta$ subset of $X$, and a characterization of these elements in the spirit of a classical theorem of Effros (e.g.,\  \cite[Theorem~3.2.4]{Gao:InvariantDescriptiveSetTheory}) was provided.

The Effros theorem is a cornerstone in the study of the structure of orbits for Polish group actions, particularly when one needs to determine whether there exist co-meagre orbits.
The topometric version obtained in \cite{BenYaacov-Melleray:Grey} left open the question of whether a weaker condition on $x$ was sufficient to establish that $x$ is metrically generic, as well as the issue of giving a criterion for the existence of metrically generic elements.
The purpose of this note is to address those two points, proving the following (we refer the reader to the beginning of the next section for a reminder of topometric conventions and notations).

\begin{thm*}[see \autoref{thm:NonMeagreOrbitClosure} and \autoref{thm:ExistenceOfGenerics} below]
  Let $(X,\tau,\rho)$ be an adequate Polish topometric space, and $G$ be a Polish group acting continuously and isometrically on $X$. Assume that the action $G \curvearrowright X$ is topologically transitive. Then:
  \begin{itemize}
  \item an element $x \in X$ is metrically generic if, and only if, $(Ux)_{\rho< \varepsilon}$ is somewhere-dense for each open $U \ni 1$ and each $\varepsilon >0$.
  \item There exists a metrically generic element if, and only if, for any neighbourhood $V$ of $1$, any $\varepsilon >0$ and any nonempty open $U \subseteq X$, there exists a nonempty open $U' \subseteq U$ such that for any nonempty open $W_1,W_2 \subseteq U'$ one has $\rho(VW_1,W_2) \le \varepsilon$.
  \end{itemize}
\end{thm*}

\section{Adequate distance and generic elements}

We allow distances to take the value $\infty$, with the convention that $r+ \infty= \infty$ for all $r \in [0,\infty]$.

\begin{conv}
  \label{conv:SpaceTopologyMetric}
  When $X$ is a set endowed with a topology $\tau$ and a metric $\rho$, the vocabulary of general topology refers to $(X,\tau)$, unless explicitly qualified, while the vocabulary of metric spaces refers to $(X,\rho)$.
  Thus, for example, $(X,\tau,\rho)$ is \emph{Polish} if $(X,\tau)$ is, and \emph{complete} if $(X,\rho)$ is.
  Similarly, a \emph{continuous and isometric} action of a group $G$ on $(X,\tau,\rho)$ is a continuous action of $G$ on $(X,\tau)$ such that each map $x \mapsto g x$ is an isometry for $(X,\rho)$.

  We denote the (topological) closure, as usual, by $\overline{A}$, and the metric closure by the (qualified) variant $\overline{A}^\rho$.
\end{conv}

\begin{dfn}
  Given a distance $\rho$ on a set $X$, $U$ a subset of $X$, and $r >0$, we let
  \begin{gather*}
    (U)_{\rho<r} = \bigl\{ x \in X \colon \rho(x,U) < r \bigr\},
    \qquad
    (U)_{\rho \le r} = \bigl\{ x \in X \colon \rho(x,U) \le r \bigr\}.
  \end{gather*}
  We call these sets \emph{thickenings} (\emph{open} and \emph{closed}, respectively) of $U$.
  We mention the distance $\rho$ in the subscript, since several distances on $X$ may be considered at the same time.
\end{dfn}

\begin{dfn}
  \label{dfn:AdequateDistance}
  Let $(X,\tau)$ be a Polish space and $\rho$ a distance on $X$ (possibly incompatible with the topology $\tau$).
  \begin{enumerate}
  \item The distance $\rho$ is \emph{adequate} if for every open set $O \subseteq X$ and $r > 0$, the thickening $(O)_{\rho<r}$ is again open.
  \item An open set $W \subseteq X$ is \emph{$\varepsilon$-small} if for every open non-empty $W_1,W_2 \subseteq W$ we have $\rho(W_1,W_2) < \varepsilon$.
  \item A point $x \in X$ is \emph{$\rho$-generic} if for every $\varepsilon > 0$, the set $(x)_{\rho<\varepsilon}$ is somewhere-dense.
  \end{enumerate}
\end{dfn}

\begin{rmk}
  \label{rmk:AdequateDistanceSmall}
  Assume that $W \subseteq X$ is open and $\varepsilon$-small.
  Then for every non-empty open $O \subseteq W$, the (relative) thickening $W \cap (O)_{\rho<\varepsilon}$ is dense in $W$.

  Indeed, let $x \in W$ and let $U \ni x$ be an open neighbourhood.
  We may assume that $U \subseteq W$.
  Then $\rho(O,U) < \varepsilon$, so $(O)_{\rho<\varepsilon} \cap U \neq \emptyset$.
\end{rmk}

\begin{exm}
  \label{exm:AdequateDistanceTopometric}
  Assume that $(X,\tau,\rho)$ is a topometric space (i.e., $\rho$ refines $\tau$ and is lower semi-continuous).
  Then $\rho$ is adequate if and only if $(X,\tau,\rho)$ is adequate in the sense of \cite{BenYaacov-Melleray:Grey}.

  In the topometric case, we observe that for an open set $W \subseteq X$,
  \begin{itemize}
  \item if $\diam_\rho W < \varepsilon$, then $W$ is $\varepsilon$-small and
  \item if $W$ is $\varepsilon$-small, then $\diam_\rho W \leq \varepsilon$.
  \end{itemize}
  Indeed, one implication is immediate from the definition of $\varepsilon$-smallness.
  For the opposite direction, assume that $x,y \in W$ and $\rho(x,y) > \varepsilon$.
  By lower semi-continuity of $\rho$, there exists open neighbourhoods $x \in W_1$ and $y \in W_2$ such that $\rho(W_1,W_2) > \varepsilon$, and we may freely assume that $W_i \subseteq W$.

  Further, in that case $x$ is $\rho$-generic if and only if it is topometrically isolated, that is, $x$ belongs to the interior of $(x)_{\rho<\varepsilon}$ for all $\varepsilon >0$ (see \cite{BenYaacov:TopometricSpacesAndPerturbations}; for a more general version of this fact, valid also for non-topometric spaces, see \autoref{lem:AdequateDistance} below).
\end{exm}

\begin{dfn}
  Let $G$ be a group. A \emph{norm} on $G$ is a function $\|{\cdot}\| \colon G \to [0,\infty]$ such that
  \begin{itemize}
  \item For all $g \in G$, $\|g\|=0 \Leftrightarrow g=1$.
  \item For all $g \in G$ $\|g\|=\|g^{-1}\|$.
  \item For all $g,h \in G$ $\|gh \| \le \|g\| + \|h\|$.
  \end{itemize}
\end{dfn}

Norms correspond to left-invariant (or right-invariant, depending on the choice of convention) metrics on $G$, via the equality $\|g\|=d(g,1)$ (or $d(g,h) = \|h^{-1}g\|$).
This left-invariant distance always defines a group topology, and we say that the norm is \emph{compatible} with that topology.
Equivalently, a norm is compatible with a group topology if the family of sets
\begin{gather}
  \label{eq:Ur}
  U_r = \bigl\{ g \in G  : \|g\| < r \bigr\}
\end{gather}
is a basis of neighbourhoods for the identity.
The Birkhoff-Kakutani Theorem asserts that a topological group is metrisable if and only if it admits a compatible norm (equivalently, a compatible left-invariant distance).
In what follows, all norms (on topological groups) are implicitly assumed to be compatible.

\begin{exm}
  \label{exm:AdequateDistanceAction}
  Assume that $G$ is a metrisable topological group (e.g., a Polish group) acting continuously on $X$.
  Let $\|{\cdot}\|$ be a compatible norm on $G$, and define
  \begin{gather*}
    \rho(x,y) = \inf \bigl\{ \|g\| : gx = y \bigr\},
  \end{gather*}
  where $\inf \emptyset = \infty$.

  Then $\rho$ is an adequate distance.
  Indeed, $(A)_{\rho<r} = U_r A$, where $U_r$ is as in \autoref{eq:Ur}.
  Moreover, since $\|{\cdot}\|$ is compatible and $G$ acts continuously on $X$, $\rho$ refines the topology of $X$.

\end{exm}

In general, the distance $\rho$ of \autoref{exm:AdequateDistanceAction} need not be lower semi-continuous.
For instance, assume that $G$ is the permutation group of the integers, acting on itself by conjugation and endowed with the norm $\|\sigma \|=\inf\{2^{-n} \colon \forall i \le n \  \sigma(i)=i\}$.
Let $G_0$ denote the subgroup of all permutations which map $0$ to itself, and fix a permutation $\sigma$ with a dense conjugacy class in $G_0$; fix also $\tau \in G_0$ which is not in the conjugacy class of $\sigma$. One can pick $(\sigma_n)$ which are conjugate to $\sigma$ and converge to $\tau$; for all $n$ one has $\rho(\sigma_n,\sigma) \le 1$ (they are conjugate by an element which fixes $0$) yet $\rho(\tau,\sigma)=\infty$.

If $\overline{\rho}$ denotes the greatest lower semi-continuous function below $\rho$, we do not know if there are any reasonable conditions under which $\overline{\rho}$, or something equivalent to it in some reasonable sense, is an adequate distance function.

\begin{exm}
  \label{exm:AdequateDistanceTopometricAction}
  \autoref{exm:AdequateDistanceTopometric} and \autoref{exm:AdequateDistanceAction} can be joined as follows.
  Let $(X,\tau,\rho)$ be an adequate topometric space, and let $G$ be a metrisable topological group acting continuously and isometrically on $X$.
  Let $\|{\cdot}\|$ be a norm on $G$, and define (we use $\vee$ as infix notation for the maximum function):
  \begin{gather*}
    \rho'(x,y) = \inf \bigl\{ \|g\| \vee \rho(gx,y) : g \in G \bigr\}.
  \end{gather*}
  Thus, if $U_r = \{ g \in G : \|g\| < r \}$, then $(A)_{\rho'<r} = U_r (A)_{\rho<r} = (U_r A)_{\rho<r}$.
  Then $\rho'$ is an adequate distance refining the topology.

  If $G$ is trivial, we obtain \autoref{exm:AdequateDistanceTopometric}, and if $\rho$ is the discrete $0/\infty$ distance, \autoref{exm:AdequateDistanceAction}.

  The distance $\rho'$ in this example plays an essential role in our approach here, because it combines the group topology and the metric on $X$: $\rho'(x,y)$ is small iff there exists $g$ close to $1$ (for the topology of $G$) which maps $x$ close to $y$ (according to $\rho$).
  As we saw above, even for discrete $\rho$ this distance $\rho'$ is not lower semi-continuous in general, which is why we allow non-topometric spaces in our setup here.
\end{exm}

We now state two lemmas concerning $\rho$-generic elements for adequate distances.
We do \emph{not} assume any compatibility between $\rho$ and $\tau$ besides adequacy.
These will prove useful in the proof of our topometric version of Effros's theorem.

\begin{lem}
  \label{lem:AdequateDistance}
  Let $\rho$ be an adequate distance on $(X,\tau)$.
  \begin{enumerate}
  \item
    \label{item:AdequateDistanceThickening}
    For every $A \subseteq X$ we have $(\overline{A})_{\rho<\varepsilon} \subseteq \overline{(A)_{\rho<\varepsilon}}$ and $(\overline{A}^\circ)_{\rho<\varepsilon} \subseteq \overline{(A)_{\rho<\varepsilon}}^\circ$.
  \item If $x$ is $\rho$-generic then $x \in \overline{(x)_{\rho<\varepsilon}}^\circ$ for every $\varepsilon > 0$.
    Moreover, for every $\varepsilon > \delta > 0$ there exists an open neighbourhood $W \ni x$ such that $(W)_{\rho<\delta} \subseteq \overline{(x)_{\rho<\varepsilon}}^\circ$.
  \end{enumerate}
\end{lem}
\begin{proof}
  For \autoref{item:AdequateDistanceThickening}, let $U = X \setminus \overline{(A)_{\rho<\varepsilon}}$.
  Since $\rho$ is adequate, $(U)_{\rho<\varepsilon}$ is open and disjoint from $A$, hence also from $\overline{A}$, so $U$ is disjoint from $(\overline{A})_{\rho<\varepsilon}$, proving the first inclusion.
  For the second inclusion, observe that $(\overline{A}^\circ)_{\rho<\varepsilon}$ is open and contained in $(\overline{A})_{\rho<\varepsilon}$, hence in $\overline{(A)_{\rho<\varepsilon}}$, so it is contained in the interior.

  Assume now that $x$ is $\rho$-generic.
  For $\varepsilon > 0$, let $V = \overline{(x)_{\rho<\varepsilon}}^\circ$, which is non-empty by assumption on $x$.
  Then $V \cap (x)_{\rho<\varepsilon} \neq \emptyset$, so $x \in (V)_{\rho<\varepsilon}$.
  By \autoref{item:AdequateDistanceThickening}, $(V)_{\rho<\varepsilon} \subseteq \bigl( \overline{(x)_{\rho<\varepsilon}} \bigr)_{\rho<\varepsilon} \subseteq \overline{(x)_{\rho<2\varepsilon}}$.
  Since $(V)_{\rho<\varepsilon}$ is open, $x \in \overline{(x)_{\rho<2\varepsilon}}^\circ$.

  For the moreover part, let $W = \overline{(x)_{\rho<\varepsilon-\delta}}^\circ$.
  Then $x \in W$, and by \autoref{item:AdequateDistanceThickening}
  \begin{gather*}
    (W)_{\rho<\delta} \subseteq \overline{\left((x)_{\rho<\varepsilon-\delta}\right)_{\rho < \delta}}^\circ \subseteq \overline{(x)_{\rho<\varepsilon}}^\circ.
    \qedhere
  \end{gather*}
\end{proof}

\begin{lem}
  \label{lem:Rosendal}
  Let $X$ be a Polish space and $\rho$ an adequate distance.
  Then the following are equivalent:
  \begin{enumerate}
  \item For every $\varepsilon > 0$, the union of $\varepsilon$-small open subsets of $X$ is dense.
  \item The set of $\rho$-generic $x \in X$ is co-meagre.
  \item The set of $\rho$-generic $x \in X$ is dense.
  \end{enumerate}
\end{lem}
\begin{proof}
  \begin{cycprf}
  \item This argument is inspired by a similar one in unpublished lecture notes of Christian Rosendal.

    Let $\{O_n\}$ be a basis of (non-empty) open sets for $X$ and let $\delta_m \rightarrow 0$.
    For each $n,m$, let $W_{n,m} \subseteq O_n$ be open, non-empty and $\delta_m$-small.
    Then $W_m = \bigcup_n \, W_{n,m}$ is open and dense in $X$ for all $m$.

    For any $k$ such that $O_k \subseteq W_{n,m}$, the thickening $W_{n,m} \cap (O_k)_{\rho<\delta_m}$ is dense in $W_{n,m}$ (see \autoref{rmk:AdequateDistanceSmall}).
    It follows that the set
    \begin{gather*}
      D_{n,m} = \bigcap_{O_k \subseteq W_{n,m}} \, (O_k)_{\rho<\delta_m}
    \end{gather*}
    is co-meagre in $W_{n,m}$.
    Thus $D_m = \bigcup_n \, D_{n,m}$ is co-meagre in $W_m$ and therefore in $X$.
    Therefore $D = \bigcap_m \, D_m$ is co-meagre in $X$.

    Assume now that $x \in D$.
    For each $m$ we have $x \in D_{n,m}$ for some $n$, and if $O_k \subseteq W_{n,m}$, then $(x)_{\rho<\delta_m} \cap O_k \neq \emptyset$.
    It follows that $W_{n,m} \subseteq \overline{(x)_{\rho<\delta_m}}$.
    Thus $\overline{(x)_{\rho<\delta_m}}$ is somewhere-dense for every $m$, i.e., $x$ is generic.
  \item Immediate.
  \item[\impfirst]
    Fix $\varepsilon > 0$.
    If $x \in X$ is $\rho$-generic, then $x \in \overline{(x)_{\rho<\varepsilon}}^\circ$ by \autoref{lem:AdequateDistance}.
    In addition, if $W_1,W_2 \subseteq \overline{(x)_{\rho<\varepsilon}}^\circ$ are open and non-empty, then both intersect $(x)_{\rho<\varepsilon}$, so $\rho(W_1,W_2) < 2\varepsilon$.
    It follows that the union of all $2\varepsilon$-small open sets contains the generic points, and is therefore dense.
  \end{cycprf}
\end{proof}

\begin{rmk}
  \label{rmk:Rosendal}
  When $(X,\tau,\rho)$ is a topometric space, the argument is much more straightforward, namely, it suffices to take the intersection over $m$ of all unions of all $\delta_m$-small open sets.
\end{rmk}

\section{Topometric group action}
\label{sec:TopometricGroupAction}

Recall that a continuous action $G \curvearrowright X$ is \emph{topologically transitive} if $G V \cap W \neq \emptyset$ for any two non-empty open $V,W \subseteq X$.
Equivalently, if $G V$ is dense for every non-empty open set $V \subseteq X$.

\begin{lem}
  Assume that $(X,\tau,\rho)$ is a Polish topometric space, and $G \curvearrowright X$ continuously and isometrically.
  \label{lem:TopologicallyTransitiveCriterion}
  The following are equivalent:
  \begin{enumerate}
  \item The action $G \curvearrowright X$ is topologically transitive.
  \item The set of $x \in X$ whose orbit is dense, is co-meagre.
  \item There exists a point $x \in X$ such that $(G x)_{\rho<\varepsilon}$ is dense for all $\varepsilon > 0$.
  \item For every two non-empty open subsets $V,W \subseteq X$: $\rho(G V, W) = 0$.
  \item For every non-empty open subset $V \subseteq X$ and every $\varepsilon > 0$: $(GV)_{\rho < \varepsilon}$ is dense in $X$.
  \end{enumerate}
\end{lem}
\begin{proof}
  \begin{cycprf}
  \item This is classical: choose a countable basis $\{O_n\}$ and let $Y = \bigcap G O_n$.
    Then each $G O_n$ is open and dense, so $Y$ is co-meagre, and the orbit of every $x \in Y$ intersects every $O_n$.
  \item Immediate.
  \item If such an $x$ exists, then for every $\varepsilon > 0$ we may find $g,h \in G$ such that $\rho(gx, V) + \rho(hx,W) < 2\varepsilon$, so $\rho(G V, W) < 2\varepsilon$.
  \item Since $\rho(G V,W) < \varepsilon$ for every open non-empty $W$.
  \item[\impfirst]
    Let $V$ be a non-empty open subset of $X$.
    For any $\varepsilon >0$ the set $(G V)_{\rho < \varepsilon}$ is open and dense, so the intersection $\overline{G V}^\rho$ is co-meagre, and in particular, dense.
    Since $\rho$ refines the topology, $\overline{G V} = \overline{\overline{G V}^\rho} = X$.
  \end{cycprf}
\end{proof}

From this point onwards, we assume that we have the data of \autoref{exm:AdequateDistanceTopometricAction}, namely, $(X,\tau,\rho)$ is an adequate topometric Polish space and $G$ a Polish group acting on $X$ continuously and isometrically.
We fix a norm on $G$, which we assume to be bounded by $1$, and define $\rho'$ as in \autoref{exm:AdequateDistanceTopometricAction}:
\begin{gather*}
  \rho'(x,y) = \inf \bigl\{ \|g\| \vee \rho(gx,y) : g \in G \bigr\}.
\end{gather*}
As pointed out above, $\rho'$ is an adequate distance refining the topology, but $(X,\tau,\rho')$ is not necessarily a topometric space.
Note also that $\rho'$ is not, in general, $G$-invariant.

Observe that $x \in X$ is $\rho'$-generic if, and only if, $ \overline{ (U   x)_{\rho < \varepsilon}}^\circ \neq \emptyset$ for every open $U$ and $\varepsilon >0$, if and only if $x \in \overline{ (U   x)_{\rho < \varepsilon}}^\circ$ for every open $U$ and $\varepsilon >0$.

\begin{lem}
  \label{lem:TopologicallyTransitiveCriterionGeneric}
  Assume that $x \in X$ is $\rho'$-generic.
  Then $G \curvearrowright X$ is topologically transitive if and only if $(G  x)_{\rho<\varepsilon}$ is dense for all $\varepsilon > 0$.
\end{lem}
\begin{proof}
  One implication is immediate from \autoref{lem:TopologicallyTransitiveCriterion}.
  For the other, assume that $G\curvearrowright X$ is topologically transitive.
  Fix $\varepsilon >0$.

  By \autoref{lem:AdequateDistance}, $x \in V = \overline{(x)_{\rho'<\varepsilon}}^\circ$.
  Then $GV$ is dense, and
  \begin{gather*}
    GV \subseteq G \cdot \overline{(x)_{\rho'<\varepsilon}} \subseteq \overline{G \cdot (x)_{\rho'<\varepsilon}} = \overline{(G x)_{\rho<\varepsilon}}.
  \end{gather*}
  Therefore, $(G x)_{\rho < \varepsilon}$ is dense as well.
\end{proof}

\begin{lem}
  \label{lem:TopologicallyTransitiveGeneric}
  Assume that $G \curvearrowright X$ is topologically transitive.
  Then the following are equivalent:
  \begin{enumerate}
  \item For every $\varepsilon > 0$, the union of $(\rho',\varepsilon)$-small open sets in $X$ is dense.
  \item The set of $\rho'$-generic points is co-meagre.
  \item A $\rho'$-generic point exists.
  \end{enumerate}
\end{lem}
\begin{proof}
  \begin{cycprf}
  \item This is \autoref{lem:Rosendal}.
  \item Immediate.
  \item[\impfirst]
    Let $x$ be $\rho'$-generic, and let $\varepsilon > 0$.
    On the one hand, the open set $\overline{(x)_{\rho'<2\varepsilon}}^\circ$ is $(\rho',4\varepsilon)$-small, as in the proof of the last implication of \autoref{lem:Rosendal}.
    On the other hand, by \autoref{lem:AdequateDistance}, $x \in \overline{(x)_{\rho'<\varepsilon}}^\circ$ and
    \begin{gather*}
      (x)_{\rho<\varepsilon} \subseteq (x)_{\rho'<\varepsilon} \subseteq \bigl( \overline{(x)_{\rho'<\varepsilon}}^\circ \bigr)_{\rho'<\varepsilon} \subseteq \overline{(x)_{\rho'<2\varepsilon}}^\circ.
    \end{gather*}
    Consequently, $(G  x)_{\rho<\varepsilon}=G(x)_{\rho<\varepsilon}$ is contained in a union of $(\rho',4\varepsilon)$-small open sets.

    Since $x$ is $\rho$'-generic and $G \curvearrowright X$ is topologically transitive, $(G  x)_{\rho<\varepsilon}$ is dense (by \autoref{lem:TopologicallyTransitiveCriterionGeneric}), completing the proof.
  \end{cycprf}
\end{proof}

\begin{ntn}
  Given $x \in X$, we let $[x] = \overline{G  x}^\rho$.
\end{ntn}
Observe that the sets $[x]$ form a partition of $X$ into $\rho$-closed sets.

\begin{lem}
  \label{lem:GenericsClosed}
  Assume that $x$ is $\rho'$-generic and $y \in [x]$.
  Then $y$ is $\rho'$-generic.
\end{lem}
\begin{proof}
  First, fix $g \in G$ and $\varepsilon >0$, then find $\delta$ such that $0< \delta \le \varepsilon$ and for all $h \in G$ one has $\|h\| \le \delta \Rightarrow \|ghg^{-1}\| \le \varepsilon$.
  Given $y \in g(x)_{\rho'< \delta}$ there exists $h$ such that $\|h\|< \delta$ and $\rho(g^{-1}y,hx)< \delta$, i.e., $\rho(y,ghg^{-1}gx) < \delta$.
  It follows that $g  (x)_{\rho' < \delta} \subseteq (g  x)_{\rho' < \varepsilon}$, hence also $g  \overline{(x)_{\rho'<\delta}}^\circ \subseteq \overline{(g  x)_{\rho' < \varepsilon}}^\circ$.
  Hence $g  x$ is $\rho'$-generic.

  Now, let $y \in [x]$ and $\varepsilon > 0$, and find $g \in G$ such that $\rho(g  x,y) < \varepsilon$.
  Then $(g  x)_{\rho' < \varepsilon} \subseteq (y)_{\rho' < 2 \varepsilon}$, so $\overline{(y)_{\rho' < 2\varepsilon}}^\circ \ne \emptyset$, and $y$ is generic as well.
\end{proof}

We are ready to prove our main lemma, describing the structure of $\rho'$-generic elements for a topologically transitive action.

\begin{lem}
  \label{lem:GenericBackAndForth}
  \begin{enumerate}
  \item Fix $\varepsilon >0$. Assume that $x,y \in X$ are both $\rho'$-generic, and that $y \in \overline{(G  x)_{\rho<\varepsilon}}$.
    Then $y \in (G  x)_{\rho\leq\varepsilon}$.
  \item Assume that $G \curvearrowright X$ is topologically transitive.
    Then the set of $\rho'$-generic elements in $X$ is either empty, or of the form $[x] = \overline{G  x}^\rho$ (where $x$ is any $\rho'$-generic element).
  \end{enumerate}
\end{lem}
\begin{proof}
  For the first item, it is enough to show that $y \in (G  x)_{\rho \leq \varepsilon'}$ for any given $\varepsilon' > \varepsilon$.
  Let $\delta_n = 2^{-n}(\varepsilon' - \varepsilon)$ and $\varepsilon_n = \varepsilon' - \delta_n$.

  We are going to construct a (convergent) sequence $(g_n)$ of elements of $G$, whose limit will send $x$ close to $y$.

  Before going into the details, recall the notation \autoref{eq:Ur}: $U_r = \bigl\{ h \in G : \|h\| < r \bigr\}$.
  Let $O_0 = G$, and once $g_n$ has been chosen, let
  \begin{gather*}
    O_{n+1} = U_{2^{-n}} \cap g_n U_{2^{-n}} g_n^{-1} \cap g_n^{-1} U_{2^{-n}} g_n,
  \end{gather*}
  observing that this is a symmetric neighbourhood of $1$.
  Then, if $z \in X$ is $\rho'$-generic, we define
  \begin{gather*}
    F_n(z) = \overline{(O_n  z)_{\rho<\varepsilon_n}},
    \qquad
    W_n(z) = \overline{(O_n  z)_{\rho<\delta_n}}^\circ.
  \end{gather*}
  By \autoref{lem:AdequateDistance}, both of them contain $z$, and
  \begin{gather}
    \label{eq:BackAndForthInclusion}
    \bigl( W_{n+1}(z) \bigr)_{\rho<\varepsilon_n}
    = \bigl( \overline{(O_{n+1}  z)_{\rho<\delta_{n+1}}}^\circ \bigr)_{\rho<\varepsilon_n}
    \subseteq \overline{(O_{n+1}  z)_{\rho<\delta_{n+1} + \varepsilon_n}}
    = F_{n+1}(z).
  \end{gather}

  Now, to the actual construction, which will ensure that:
  \begin{itemize}
  \item For even $n$ we have $g_n^{-1}  y \in F_n(x)$ and $g_{n+1} \in g_n O_n$.
  \item For odd $n$ we have $g_n  x \in F_n(y)$ and $g_{n+1} \in O_n g_n$ (equivalently, $g_{n+1}^{-1} \in g_n^{-1} O_n$).
  \end{itemize}
  We may start with $g_0=1$, observing that, indeed, $g_0^{-1} y = y \in \overline{(G  x)_{\rho<\varepsilon}} = F_0(x)$.

  Assume that $g_n$ has been chosen, say for some even $n$.
  Then
  $$y \in g_n  F_n(x) \cap W_{n+1}(y)= \overline{\bigl(g_nO_n  x\bigr)_{\rho < \varepsilon_n}} \cap W_{n+1}(y) .$$
  Since $W_{n+1}(y)$ is open, it intersects $(g_n O_nx)_{\rho < \varepsilon_n}$.
  Together with \autoref{eq:BackAndForthInclusion}:
  \begin{gather*}
    g_n O_n  x \cap F_{n+1}(y) \supseteq g_n O_n x \cap \bigl( W_{n+1}(y) \bigr)_{\rho<\varepsilon_n} \neq \emptyset.
  \end{gather*}
  We may then choose $g_{n+1} \in g_n O_n$ such that $g_{n+1}  x \in F_{n+1}(y)$.
  The odd step is similar.

  We claim that the sequence $(g_n)$ is Cauchy for the upper uniformity on $G$, hence convergent since $G$ is Polish.
  Indeed, let us consider an even $m \geq 2$, say $m = n + 2$.
  Then $g_m \in g_{n+1} O_{n+1}$ and $g_{m+1} \in O_{n+2} g_m$, so
  \begin{gather*}
    g_{m+1}g_m^{-1} \in O_{n+2} \subseteq U_{2^{-n-1}},
    \\
    g_m^{-1}g_{m+1} \in g_m^{-1} O_{n+2} g_m \subseteq  O_{n+1} \cdot (g_{n+1}^{-1} O_{n+2} g_{n+1}) \cdot O_{n+1} \subseteq U_{2^{-n}} \cdot U_{2^{-n-1}} \cdot U_{2^{-n}}.
  \end{gather*}
  Therefore $\|g_{m+1}g_m^{-1}\| < 2^{-m+1}$ and $\|g_m^{-1}g_{m+1}\| < 5 \cdot 2^{-m+1}$, which is good enough.
  The odd case is similar.

  Let $g = \lim g_n \in G$.
  For all $n$ we have $y \in \overline{g_{2n}\bigl(O_{2n}  x\bigr)_{\rho<\varepsilon_{2n}}}$, so we may choose $z_n \in X$ and $u_n \in O_{2n}$ such that $g_{2n}z_n \rightarrow y$ and
  \begin{gather*}
    \rho(g_{2n} z_n, g_{2n} u_n x) = \rho(z_n,u_n  x) < \varepsilon_{2n} \le \varepsilon'.
  \end{gather*}
  By continuity of the group action, $g_{2n}u_n x \rightarrow g x$.
  Since $\rho$ is, in addition, lower semi-continuous, we obtain $\rho(y,gx) \le \varepsilon'$, as promised.

  For the second item, assume that $G \curvearrowright X$ is topologically transitive and that $x$ is $\rho'$-generic.
  We know by \autoref{lem:GenericsClosed} that any $y \in [x]$ is $\rho'$-generic as well.
  Conversely, since the action is topologically transitive we have $\overline{(G  x)_{\rho<\varepsilon}} = X$ for all $\varepsilon > 0$, by \autoref{lem:TopologicallyTransitiveCriterionGeneric}.
  By the first item, every $\rho'$-generic element belongs to $\bigcap_{\varepsilon > 0} \, (G  x)_{\rho \le \varepsilon} = [x]$.
\end{proof}

We obtain the following topometric version of the Effros theorem.

\begin{thm}
  \label{thm:NonMeagreOrbitClosure}
  Let $(X,\tau,\rho)$ be an adequate Polish topometric space, $G$ a Polish group acting on $X$ continuously, isometrically and topologically transitively.
  Then, for $x \in X$, the following are equivalent:
  \begin{enumerate}
  \item The orbit closure $[x] = \overline{G  x}^\rho$ is co-meagre.
  \item The set $(G  x)_{\rho<\varepsilon}$ is non-meagre for all $\varepsilon > 0$.
  \item The point $x$ is $\rho'$-generic, namely $(U x)_{\rho < \varepsilon}$ is somewhere-dense for every $\varepsilon >0$ and open $U \ni 1$.
    Equivalently, $x \in \overline{(U  x)_{\rho<\varepsilon}}^\circ$ for every $\varepsilon > 0$ and open $U \ni 1$.
  \end{enumerate}
\end{thm}
\begin{proof}
  \begin{cycprf}
  \item
    Immediate since $\overline{G  x}^\rho$ is contained in $(G  x)_{\rho<\varepsilon}$ for all $\varepsilon >0$.
  \item
    For any open $U \ni 1$ we can express $G$ as $\bigcup_n g_n U$.
    Since $(G  x)_{\rho<\varepsilon}$ is not meagre, neither is $(U  x)_{\rho<\varepsilon}$, so it is somewhere-dense.
    In other words, $x$ is $\rho'$-generic.
    The second characterisation is by \autoref{lem:AdequateDistance}.
  \item[\impfirst]
    Let $X_0 \subseteq X$ consist of all $\rho'$-generic elements, and assume that $x \in X_0$.
    By \autoref{lem:TopologicallyTransitiveGeneric}, the set $X_0$ is co-meagre.
    By \autoref{lem:GenericBackAndForth}, $X_0 = [x]$.
  \end{cycprf}
\end{proof}

\begin{rmk}
  \label{rmk:Grey}
  If we assume that $G  x$ is dense in $X$, the above conditions are also equivalent, by \cite[Theorem~5.2]{BenYaacov-Melleray:Grey}, to the following conditions:
  \begin{enumerate}
    \setcounter{enumi}{3}
  \item $\overline{G  x}^\rho$ is $G_\delta$.
  \item For any open subset $U$ of $G$ and any $\varepsilon >0$, $\bigl(U  x\bigr)_{\rho < \varepsilon}$ is open in $\overline{G  x}^\rho$.
  \item For any open subset $U$ of $G$ and any $\varepsilon >0$, $\bigl(U  x \bigr)_{\rho < \varepsilon} \cap G x$ is open in $G  x$.
  \end{enumerate}
\end{rmk}

The criterion obtained in this paper provides a condition on $x$ that is sufficient for $\overline{G  x}^\rho$ to be co-meagre, but is seemingly weaker than the conditions from \cite{BenYaacov-Melleray:Grey}.
This approach also enables us to state a criterion for the existence of such points.
The analogous statement for Polish group actions on Polish spaces is due to Rosendal and part of our arguments here are adaptations of Rosendal's proof to the topometric setting.

\begin{thm}
  \label{thm:ExistenceOfGenerics}
  Let $(X,\tau,\rho)$ be an adequate Polish topometric space, $G$ a Polish group acting on $X$ topologically transitively.
  Then the following are equivalent:
  \begin{enumerate}
  \item
    There exists $x \in X$ such that $[x] = \overline{G  x}^\rho$ is co-meagre.
  \item
    For any $\varepsilon >0$ the union of all $(\rho',\varepsilon)$-small open sets is dense.
  \item
    For any open $V \ni 1$, any $\varepsilon >0$ and any non-empty open $U \subseteq X$, there exists a non-empty open $U' \subseteq U$ such that for any non-empty open $W_1,W_2 \subseteq U'$ one has $\rho(V  W_1,W_2) \le \varepsilon$.
  \end{enumerate}
\end{thm}
\begin{proof}
  \begin{cycprf}
  \item[\eqnext]
    By \autoref{thm:NonMeagreOrbitClosure} there exists $x \in X$ such that $[x]$ is co-meagre iff there exists a $\rho'$-generic $x$, and then \autoref{lem:TopologicallyTransitiveGeneric} states that the two conditions are equivalent.
  \item[\eqnext]
    The union of all $(\rho',\varepsilon)$-small open sets is dense if and only if every non-empty open set $U$ contains a $(\rho',\varepsilon)$-small one $W$, namely such that, if $W_1,W_2 \subseteq W$ are non-empty and open, then $(W_1)_{\rho'<\varepsilon} \cap W_2 \neq \emptyset$.
    Since $(W_1)_{\rho'<\varepsilon} = (U_\varepsilon W_1)_{\rho<\varepsilon}$, where $U_\varepsilon= \bigl\{g \in G \colon \|g\|< \varepsilon\bigr\}$ as in \autoref{eq:Ur}, we obtain the alternate formulation.
  \end{cycprf}
\end{proof}

Let us give an example of an application of our topometric version of the Effros theorem.
Recall that a \emph{Polish topometric group} $(G,\tau,\rho)$ is a topometric space such that $(G,\tau)$ is a topological group, and the distance $\rho$ is invariant under both left and right translation (see \cite{BenYaacov-Berenstein-Melleray:TopometricGroups}).
Given a Polish topometric group $(G,\tau,\rho)$ and $n < \omega$, we turn $G^n$ into an adequate Polish topometric space by endowing it with the product topology and the metric $\rho(x, y)= \max_{i<n} \rho(x_i,y_i)$.
The group $G$ acts on each $G^n$ by diagonal conjugation.
We say that $x \in G^n$ is \emph{metrically generic} if $\overline{G \cdot x}^\rho$ is co-meagre in $G^n$ (here we use $\cdot$ to denote the action of $G$ by diagonal conjugation, to avoid confusion with the product of elements of $G$), or equivalently, if $(G \cdot x)_{\rho < \varepsilon}$ is co-meagre for every $\varepsilon > 0$.
We say that $(G,\tau,\rho)$ has \emph{ample metric generics} if $G^n$ admits a metrically generic point for each $n$.

\begin{prp}\label{prp:criterion}
  Let $(G,\tau_G,\rho_G)$ and $(H,\tau_H,\rho_H)$ be two Polish topometric groups and $\varphi \colon H \to G$ a group homomorphism such that:
  \begin{itemize}
  \item $\varphi \colon (H,\tau_H) \to (G,\tau_G)$  and $\varphi \colon (H,\rho_H) \to (G,\rho_G)$ are continuous.
  \item For any open $U$ in $H$ and any $\varepsilon >0$, $(\varphi(U))_{\rho_G < \varepsilon}$ is open in $G$ (i.e.~ $\varphi$ is \emph{topometrically open} in the sense of \cite{BenYaacov-Melleray:Grey}).
  \item $\varphi$ has dense image .
  \end{itemize}
  Assume further that $(H,\tau_H,\rho_H)$ has ample metric generics.

  Then $(G,\tau_G,\rho_G)$ has ample metric generics, and images of metrically generic elements of $H^n$ are metrically generic elements of $G^n$.
\end{prp}
\begin{proof}
  Since $H$ has ample generics, each action $H \curvearrowright H^n$ is topologically transitive.
  Since $\varphi$ has dense image, each action $G \curvearrowright G^n$ is topologically transitive as well.

  Assume that $x \in H^n$ is a metric generic.
  Let $U$ be an open neighbourhood of $1$ in $G$ and $\varepsilon > 0$.
  Let also $y = \varphi(x) \in G^n$.

  Find an open neighbourhood $V$ of $1$ in $H$ and $\delta > 0$ such that $\varphi\bigl( (V \cdot x)_{\rho_H < \delta} \bigr) \subseteq (U \cdot y)_{\rho_G < \varepsilon}$.
  Using \autoref{lem:AdequateDistance} and the continuity of $\varphi$, we have
  \begin{gather*}
    \overline{(U \cdot y)_{\rho_G < 2 \varepsilon}}
    \supseteq \overline{ \Bigl( \varphi\bigl( (V \cdot x)_{\rho_H < \delta} \bigr) \Bigr)_{\rho_G < \varepsilon}}
    \supseteq \Bigl( \overline{ \varphi\bigl( (V \cdot x)_{\rho_H < \delta} \bigr) } \Bigr)_{\rho_G < \varepsilon}
    \supseteq \Bigl( \varphi\bigl( \overline{ (V \cdot x)_{\rho_H < \delta} } \bigr) \Bigr)_{\rho_G < \varepsilon}.
  \end{gather*}
  By the characterization of metric generics we know that $\overline{(V \cdot x)_{\rho_H < \delta}}$ contains some non-empty open $W \subseteq H^n$.
  It follows that $\overline{(U \cdot y)_{\rho_G < 2 \varepsilon}}$ contains $(\varphi(W))_{\rho_G < \varepsilon}$, and we are done since $\varphi$ is topometrically open.
\end{proof}

This in particular recovers the sufficient condition for ample generics in Polish topometric groups given in \cite{BenYaacov-Berenstein-Melleray:TopometricGroups} (there $H$ is endowed with the discrete metric, $H$ is a subgroup of $G$ and $\varphi$ is the identity on $H$). We note, however, that Proposition \ref{prp:criterion} could also be obtained via results obtained in \cite{BenYaacov-Berenstein-Melleray:TopometricGroups} or \cite{BenYaacov-Melleray:Grey}. The reason why, at the moment, we cannot present a more convincing application of our topometric Effros theorem is that we lack examples of adequate topometric spaces, and the applicability of our result is still quite limited in practice, though we hope this will change.

We conclude this paper by discussing a potential source of interesting examples.

Fix a homogeneous metric structure $\cM$ and a countable group $\Gamma$ (say, generated by some finite set $S$ for ease of exposition), and consider the space $A(\Gamma,\cM)$ of all actions of $\Gamma$ on $\cM$.
One can see $A(\Gamma,\cM)$ as a closed subspace of $\Aut(\cM)^\Gamma$, so the induced topology $\tau$ turns $A(\Gamma,\cM)$ into a Polish space.
The group $\Aut(\cM)$ and the space $A(\Gamma,\cM)$ also carry natural metrics, namely, for $g,h \in G$ and $\alpha,\beta \in A(\Gamma,\cM)$:
\begin{gather*}
  d_u(g,h)=\sup \bigl\{ d\bigl( g(x),h(x) \bigr) : x \in M \bigr\},
  \quad
  \rho(\alpha,\beta)= \sup \Bigl\{ d_u\bigl(\alpha(s),\beta(s) \bigr) : s \in S \Bigr\}.
\end{gather*}

Note that $(A(\Gamma,\cM),\tau,\rho)$ is a Polish topometric space; and $\Aut(\cM)$ acts by conjugation on $A(\Gamma,\cM)$, via $(g \cdot \alpha) (\gamma)= g \alpha(\gamma) g^{-1}$.
In some concrete settings (e.g., in ergodic theory, see \cite{Berenstein-Ibarlucia-Henson:ECMPAFG}) one would like to know if there exists a privileged action of $\Gamma$ on $\cM$.
In many cases, the orbits for this $\Aut(\cM)$-action are meagre, and the next-best thing would be the existence (and description) of a metric generic.
Thus it would be interesting to apply our criterion in this setting, i.e., to know that $A(\Gamma,\cM)$ is adequate.

Note that when $\Gamma$ is the free group generated by $S$, $A(\Gamma,\cM)$ is isomorphic, as a topometric space, to $\Aut(\cM)^n$, hence is adequate.

\begin{qst}
  Let $\cM$ be the Urysohn metric space, and $\Gamma$ a countable group.
  Is the space of actions $A(\Gamma,\cM)$ an adequate Polish topometric space?
\end{qst}

Of course one could replace $\cM$ by any other homogeneous metric structure and ask the same question (in particular, one could consider the standard atomless measure algebra, or the separable Hilbert space).

\providecommand{\bysame}{\leavevmode\hbox to3em{\hrulefill}\thinspace}

\end{document}